\documentclass{article}
\usepackage{amsmath, amsthm, amssymb, amsfonts}
\usepackage{thmtools}
\usepackage{graphicx}
\usepackage{setspace}
\usepackage{geometry}
\usepackage{float}
\usepackage{hyperref}
\usepackage[utf8]{inputenc}
\usepackage[english]{babel}
\usepackage{framed}
\usepackage[dvipsnames]{xcolor}
\usepackage{tcolorbox}
\usepackage{ctex}  

\newtheorem{theorem}{Theorem}[section]

\newtheorem{definition}{Definition}[section]
\newtheorem{remark}{Remark}[section]

\allowdisplaybreaks[3]

\colorlet{LightGray}{White!90!Periwinkle}
\colorlet{LightOrange}{Orange!15}
\colorlet{LightGreen}{Green!15}

\declaretheoremstyle[name=Theorem,]{thmsty}

\begin{document}
	\title{The central limit theorems for integrable Hamiltonian systems perturbed by white noise} 
	
	\author{Chen Wang $^{a,1}$, Yong Li $^{*,b,1,2}$ }
	
	\renewcommand{\thefootnote}{}
	\footnotetext{\hspace*{-6mm}
		
		\begin{tabular}{l l}
			$^{*}$~~~The corresponding author.\\
			$^{a}$~~~E-mail address : wangchen22@mails.jlu.edu.cn.\\
			$^{b}$~~~E-mail address : liyong@jlu.edu.cn.\\
			$^{1}$~~~School of Mathematics, Jilin University, Changchun 130012, People's Republic of China.\\
			$^{2}$~~~School of Mathematics and Statistics, Center for Mathematics and Interdisciplinary Sciences,\\
			~~~~~ Northeast Normal University, Changchun 130024, People's Republic of China.
	\end{tabular}}
	
\date{}     

\maketitle

\noindent$\mathbf{Abstract.}$ In this paper, we consider the dynamics of
integrable stochastic Hamiltonian systems. Utilizing the Nagaev-Guivarc'h method, we obtain several generalized results of the central limit theorem.
Making use of this technique and the Birkhoff ergodic theorem, we prove that the invariant tori persist under stochastic perturbations. Moreover, they asymptotically follow a Gaussian distribution, which gives a positive answer to the stability of  integrable stochastic Hamiltonian systems over time. Our results hold true for both Gaussian and non-Gaussian noises, and their intensities can be not small.

\noindent $\mathbf{Keywords.}$  Integrable stochastic Hamiltonian system; Central limit theorem; Invariant tori
\section{Introduction}

This paper mainly concerns the stability of invariant tori for integrable Hamiltonian systems under white noise perturbations. Such white noise might be Gaussian or non-Gaussian with significant intensities. 

Hamiltonian systems, endowed with significant geometric structures, constitute a crucial class of dynamical systems frequently used to depict the motion of planets and the evolution of numerous microscopic systems \cite{05,18,19}. Within this vast kind of models, integrable systems stand out as a special subclass which emerge from simple physical models such as two-body problems, linear oscillators, and so on. These systems admit an explicit description in terms of action-angle variables, where each phase space is foliated into a family of invariant tori, and the movement of the systems on each torus is quasi-periodic.

However, during the study of the three-body problem, Poincar$\mathrm{\acute{e}}$ \cite{27} realized that the first integrals of the systems don't \color{black} always exist and most of systems are \color{black} not integrable in fact. Based on this, he named the study of nearly integrable systems \textquotedblleft  Fundamental Problem of Dynamics\textquotedblright . The celebrated KAM theorem (see \cite{28}) tells us that under proper non-degenerate conditions, the invariant tori of such systems will survive after the addition of deterministic perturbations, for some developments, see, for example, \cite{25,08,11}. Such stability can be called as persistence. \color{black}

Certainly, when constructing system models using deterministic differential equations, it is imperative that the parameters within these equations are also deterministic. However, in almost every aspect of our real world, pinpointing the precise values of these parameters poses a significant challenge, due to their inevitable perturbations by the environment across both temporal and spatial dimensions.
Therefore, as Bismut said in \cite{03}, it is reasonable to consider stochastic differential equations with Hamiltonian structures for our study. In the selection of noise types, white noise stands out as a viable option due to its extensive applications in quantum physics, finance, and so on. Specifically, Gaussian white noise holds a basic position representing random factors in a continuous manner, mathematically expressed through the generalized time derivatives of Brownian motions. In this aspect, we refer to some relevant literature \cite{31,21,33}. In recent years, there has been a growing interest in analyzing non-Gaussian white noise, with a particular emphasis on $\mathrm{L\acute{e}vy}$ white noise, as stochastic differential equations driven by $\mathrm{L\acute{e}vy}$ processes can describe physical systems where continuous and jumping random factors interact. However, despite this increasing attention, there have been relatively few results regarding this type of Hamiltonian system, see \cite{02,30}.

Going further, stability remains an important consideration in dynamical systems, and this holds true for stochastic Hamiltonian systems as well, where the addition of external noise can have significant impacts on the original systems, especially for integrable systems. 
Notably, we pose the following questions: 

\textit{After adding stochastic perturbations to  integrable Hamiltonian systems, 
	what properties will be preserved over time? Do central limits exist for these perturbed systems?  If they do exist, how can we estimate these limits?}

To this end, 
some central limit theorems, as exemplified in \cite{26}, should be taken into account. These theorems tell us that, under certain conditions, the standardized sum of random variables tends to a normal distribution as the sample size increases. Reviewing the existing literature, limit theorems for dynamical systems have aroused the interest, which do describe the long-term behaviors of orbits that reach stable limits beyond chaos. Let us review some of that work.
In 1989, Denker \cite{06} dealt with 
the properties of the subspace in $L^{2}$ consisting of those functions for which the central limit theorem  holds and provided some applications.
Later, Freidlin and Wentzell discussed in their book \cite{15} about the averaging principle, fast oscillating perturbations, and so on. 
Chen \cite{04} investigated the sharp limit theorems for functionals of ergodic Markov chains on general measurable state spaces, which include the central limit theorem, the law of iterated logarithm and the moderate deviation principle.
In 2004, Dolgopyat \cite{09} considered a large class of partially hyperbolic systems and proved that as long as the rate of mixing is sufficiently high, the systems  can \color{black} satisfy many classical limit theorems. In \cite{29}, Shirikyan showed that once the distribution of perturbation  is sufficiently non-degenerate, strong law of large numbers and central limit theorem for solutions of a class of dissipative partial differential equations can  be established, and the corresponding rates of convergence are estimated as well. Derriennic \cite{07} provided an introduction to some previous developments of limit theorems in ergodic theory. In 2012, Komorowski \cite{20} showed the law of large numbers and central limit theorem for an additive functional with respect to a non-stationary Markov process, provided that the dual transition probability semigroup, defined on measures, is \color{black} strongly
contractive in an appropriate Wasserstein metric. Recently, Simoi et al. \cite{10} proved several limit theorems for a simple class of partially hyperbolic fast–slow systems based on a mixture of standard pairs and transfer operators, Huang et al. \cite{17} considered the one-dimensional stochastic heat equations driven by multiplicative space–time white
noise and proved that after renormalization, the spatial integral of the solution will converge in total variance distance to a standard normal distribution together with a functional
version of central limit theorems. Through the
techniques of time discretization and truncation, Liu et al. \cite{24} studied the averaging principle for stochastic differential equations with slow and fast time-scales,
where the coefficients in the slow equation are time dependent. 
In 2022, Wang et al. \cite{32} established the strong law of large numbers and the central limit theorem for additive functionals of the segment processes and derived the law of iterated logarithm under some dissipative assumptions, at the same year, Fernando and P$\mathrm{\grave{e}}$ne \cite{14} studied higher order expansions both in the Berry–Ess$\mathrm{\acute{e}}$en estimate and the local limit theorems for Birkhoff sums of chaotic probability preserving dynamical systems,  Liu and Lu \cite{23} proved the strong law of large numbers and central limit theorems, alongside estimating the rates of convergence, for the continuous time solution processes of the incompressible 2D Navier-Stokes equations with deterministic time quasi-periodic forces and white noise. 

In the present article, we adopt the idea of
characteristic functions in \cite{13} and establish several central limit theorems for integrable stochastic Hamiltonian systems perturbed by white noise, where the Birkhoff ergodic theorem and some auxiliary probabilistic results derived in this paper are used. Successfully, the invariant tori are shown to persist asymptotically under stochastic perturbations, so that within the framework of central limit theorem, such kind of integrable systems are stable as time goes to infinity. Moreover, we underline that our preconditions do not require properties like exponential mixing, dissipation, etc., which gave rise to high frequencies in some existing literature.

This paper is organized as follows. In Section 2, we recall some preliminaries for later use.
In Section 3, referring to \cite{13} and \cite{16}, we derive several results about the central limit theorem. To be specific, in the theory of dynamical systems, it's impossible to ensure that the distribution of a solution process remains consistent over time, which motivates us to generalize Theorem 1.2 of \cite{13},
so as to pave the way for the upcoming discussions. In Section 4, we provide applications of Section 3: the central limit theorems for integrable Hamiltonian systems perturbed by white noise. 

Throughout this paper, we equip $\mathbb{R}^{d}$ with the general Euclidean inner product $\left \langle  \cdot,\cdot \right \rangle $ and Euclidean norm $\left \| \cdot \right \| $.
For later calculations, we denote $z^{2}$ as a $d\times d$ matrix $zz^{\mathrm{T}}$ for any $z=\left ( z_{1},\cdots,z_{d} \right )^{\mathrm{T}} \in \mathbb{R}^{d}$. Moreover, we denote $\dot{x}=dx/dt$ as the derivative of $x$ with respect to time $t$ and use the notation $\mathcal{N}\left ( \mu,\sigma^{2} \right ) $ to represent the Gaussian distribution with expectation $\mu$ and covariance $\sigma^{2}$.

\section{Preliminaries}
\begin{definition}
	Given a filtered probability space $\left ( \Omega,\mathcal{F},\mathcal{F}_{t},\mathbb{P} \right ) $, $\left \{ X\left (t  \right ),t\ge 0  \right \} $ is said to be an $\mathbb{R}^{d}$-valued L$\mathrm{\acute{e}vy}$ process, if the following conditions are satisfied.
	
	\noindent (i) For any $n \ge 1$ and $0 \le t_{0} <t_{1}<\cdots<t_{n}$, random variables $X\left ( t_{0} \right ) $, $X\left (t_{1}  \right )-X\left ( t_{0} \right )  $, $\cdots$, $X\left (t_{n} \right ) -X\left (   t_{n-1}\right )$ are independent of each other.
	
	\noindent (ii) $X\left ( 0 \right )=0 ,\ \mathrm{a.s.}.$
	
	\noindent (iii) The distribution of $X\left (s+t  \right )-X\left (s  \right )  $ does not depend on $s\ge 0$.
	
	\noindent (iv) For any $a>0$ and $s \ge 0$,
	\begin{align*}
		\lim_{t \to s}\mathbb{P}\left ( \left |X\left ( t \right ) -X\left ( s \right )   \right | >a \right ) =0.
	\end{align*}
\end{definition}

\begin{theorem} (Continuity theorem) 
	In order that a sequence $\left \{ F_{n} \right \} $ of probability distributions converges properly to a probability distribution $F$, it 
	is necessary and sufficient that the sequence $\left \{ \varphi_{n} \right \} $ of their characteristic functions converges pointwise to a limit $\varphi$, and that $\varphi$ is continuous in some neighborhood of the origin. In this case $\varphi$ is the characteristic function of $F$.
\end{theorem}
\begin{theorem}($L\acute{e}vy-It\hat{o}$ decomposition) If $\left \{ X\left (t  \right ),t\ge 0  \right \} $ is a $d$-dimensional $L\acute{e}vy$ process, then $X\left ( t \right ) $ has a decomposition
	\begin{align*}
		X\left ( t \right ) =\gamma t+\xi B\left ( t \right )+\int_{\left |z  \right |<1 } z\tilde{N}\left ( t,dz \right ) +\int_{\left |z  \right |\ge 1 } zN\left ( t,dz \right ) ,\tag{2.1}
	\end{align*}
	satisfying $\int\min\left \{\left |z \right |^{2}  ,1  \right \} \nu\left ( dz \right )  <\infty$, where $\gamma \in \mathbb{R}^{d}$ is the drift vector, $B\left ( t \right ) $ is the standard $d$-dimensional pairwise independent Brownian motion, $\xi$ is a positive constant, $N\left ( ds,dz \right ) $ is the Poisson random measure of $X\left ( t \right ) $ defined as follows. For a measurable space $\left ( E,\Upsilon \right ) $ and a filtered probability space $\left ( \Omega,\mathcal{F},\mathcal{F}_{t},\mathbb{P} \right ) $,
	\begin{align*}
		N\left ( t,Q \right )\left ( w \right )  :=\sharp\left \{ s\mid 0 \le s \le t:\Delta X\left ( s \right )\left ( w \right ) 
		=X\left ( s \right )\left ( w \right ) 
		-X\left ( s- \right ) \left ( w \right ) \in Q 
		\right \}.
	\end{align*} 
	Here $Q \in \mathcal{B}\left ( \mathbb{R}^{d} \setminus \left \{0  \right \}  \right )$, $w \in \Omega$ and 
	$\sharp\left \{ \cdot \right \} $ is the cardinal number of set $\cdot$. Meanwhile, by denoting the jump measure of $N$ 
	as $\nu=\mathbb{E}N\left ( 1,\cdot \right ) $, where $\mathbb{E}$ is the expectation with respect to $\mathbb{P}$, we define the compensated Poisson random measure $\tilde{N}$ of $X\left ( t \right ) $ as $\tilde{N}\left (  ds,dz\right ) =N\left ( ds,dz \right )-\nu\left ( dz \right )ds $.
\end{theorem}
\begin{remark}
	It is evident that any $L\acute{e}vy$ process can be readily decomposed into three distinct components: a Brownian motion with a deterministic drift, small jumps, and large jumps. This decomposition can be described using the parameters $\left ( \gamma,\xi,N,\nu\right ) $. If $\nu=N=0$, then (2.1) reduces to Wiener process.
\end{remark}
\begin{theorem} (Liouville-Arnold theorem)
	Let $\left ( M, \omega \right )$ be a $2n$-dimensional symplectic manifold, and let $H: M \to \mathbb{R}$ be a smooth Hamiltonian function defining a Hamiltonian system on $M$. Suppose there exist $n$ functionally independent smooth functions $F_1, \cdots, F_n: M \to \mathbb{R}$ such that $\left \{ F_i, F_j \right \} = 0$ for all $i, j = 1, \ldots, n$, where $\left \{\cdot, \cdot  \right \}$ denotes the Poisson bracket. Then, for any regular value $f = \left ( f_1, \cdots, f_n \right )$ of the map $F = \left ( F_1, \cdots, F_n \right ): M \to \mathbb{R}^n$, the level set  
	\begin{align*}
		M_f = \left \{ x \in M : F_i(x) = f_i, \, i = 1, \cdots, n \right \}
	\end{align*}
	is an $n$-dimensional smooth submanifold of $M$, there exists a neighborhood $U$ of $M_f$ in $M$ and a symplectomorphism $\Phi: U \to V$ (where $V$ is an open subset of $\mathbb{T}^n \times \mathbb{R}^n$ with $\mathbb{T}^n$ the $n$-torus) such that in the new coordinates $\left(\phi_1, \cdots, \phi_n, I_1, \cdots, I_n\right )$ on $V$, the transformed Hamiltonian system takes the form  
	\begin{align*}
		\dot{\phi}_i = \omega_i(I), \quad \dot{I}_i = 0, \quad i = 1, \cdots, n,
	\end{align*}
	where $\omega_i\left ( I \right ) $ are smooth functions of $I = \left ( I_{1},\cdots,I_{n} \right )$ alone. Here, the functions $\phi_{i}$ and $I_{i}$ are called angle variables and action variables, respectively.
\end{theorem}

\section{Some results about the central limit theorem 
}
In this section, adopting the viewpoint of the characteristic function, we present several conclusions about the central limit theorem. 
\subsection{One-dimensional discrete time case}
Consider a sequence of real \color{black} random variables $\left \{Y_{i},i=1,\cdots,n  \right \} $ in $L^{2}$ and denote $Q_{n}$ as its partial sums.
\begin{theorem}
	Assume that there exist $\delta>0$ and functions $c\left ( t \right ) $, $\lambda_{k}\left ( t \right ) $ and $d_{n}\left ( t \right ) $ defined on $\left [-\delta,\delta  \right ] $, such that for all $t \in \left [-\delta,\delta  \right ]$ and $n \in \mathbb{N}$,
	\begin{align*}
		E\left (\mathrm{exp}\left \{ itQ_{n} \right \}\right )=\prod_{k=1}^{h }c\left ( t \right )\lambda_{k}\left ( t \right )^{m_{k}} +d_{n}\left ( t \right ) ,
	\end{align*}
	where $h=h\left ( n \right ) $ is a positive integer and $\left \{ m_{k} \right \}_{k=1,\cdots,h} $ satisfies $
	m_{1}+\cdots+m_{h}=n$.
	Moreover, assume that
	
	\noindent (i) There exist $A_{k}$ and $\sigma_{k}^{2}$, $k=1,\cdots,h$, in $\mathbb{R}$, such that when $t \to 0$,
	\begin{align*}
		\lambda_{k}\left ( t \right )=\mathrm{exp}\left \{ iA_{k}t-\frac{\sigma_{k}^{2}t^{2}}{2}+o\left ( t^{2} \right )  \right \} .
	\end{align*}

	\noindent (ii) The function $c$ is continuous at 0.
	
	\noindent (iii) The quantity $\left \| d_{n} \right \| _{L^{\infty}\left ( -\delta,\delta \right ) }$ tends to 0 when $n$ tends to infinity,
	
	\noindent and one of the following two additional assumptions is set up.
	
	\noindent (iv) The quantity $\sum_{k=1}^{h}m_{k}\sigma^{2}_{k}/n$ converges to a limit $\mu^{2}\ge0$ as $n$ goes to infinity.
	
	\noindent (v) The quantity $\sum_{k=1}^{h}m_{k}\sigma^{2}_{k}/n$ is bounded by $\eta^{2}\ge 0$.
	
	\noindent Then when $n$ tends to infinity, under (i)-(iv), $\left (Q_{n}-\sum_{k=1}^{h}m_{k}A_{k}  \right )/\sqrt{n}$ converges to a Gaussian distribution $\mathcal{N}\left (0,\mu^{2}  \right ) $, while under (i)-(iii) and (v), there exists a constant $ \zeta \in\left [0,\eta  \right ] $ and a sequence $\left \{ n_{j} \right \}_{j \in \mathbb{N}}\subset \left \{ n \right \}$, such that $\left (Q_{n_{j}}-\sum_{k=1}^{h_{j}}m_{k}A_{k}  \right )/\sqrt{n_{j}}$ converges to a Gaussian distribution $\mathcal{N}\left (0,\zeta^{2} \right ) $, where $h_{j}=h\left ( n_{j} \right ) $ is a positive integer.
\end{theorem}
\begin{proof}
	First, taking $t = 0$, we obtain that $\prod_{k=1}^{h}c\left ( 0 \right ) \lambda_{k}\left ( 0 \right )^{m_{k}} =1-d_{n}\left ( t \right )\to 1 $ while $\lambda_{k}\left (0 \right ) =1$ for $k=1,\cdots,h$, then $c\left ( 0 \right )=1 $. Next, by considering $t/\sqrt{n}$ instead of $t$, we obtain from (i) that
	\begin{align*}
		&\lambda_{k}\left ( \frac{t}{\sqrt{n}} \right )=\mathrm{exp}\left \{ i \frac{tA_{k} }{\sqrt{n}}- \frac{t^{2}\sigma^{2}_{k}}{2n} +o\left ( \frac{t^{2}}{n} \right )  \right \} ,
	\end{align*}
	so that
	\begin{align*}
		E\left (\mathrm{exp}\left \{ \frac{i tQ_{n}}{\sqrt{n}} \right \}\right )&=\prod_{k=1}^{h}c\left ( \frac{t}{\sqrt{n}} \right )\lambda_{k}\left ( \frac{t}{\sqrt{n}} \right )^{m_{k}} +d_{n}\left ( \frac{t}{\sqrt{n}} \right )\\
		&=\prod_{k=1}^{h}c\left ( \frac{t}{\sqrt{n}} \right )\cdot \mathrm{exp}\left \{ i\frac{tA_{k}m_{k}}{\sqrt{n}}-\frac{t^{2}\sigma^{2}_{k}m_{k}}{2n}+o\left ( \frac{t^{2}m_{k}}{n} \right )  \right \}+d_{n}\left ( \frac{t}{\sqrt{n}} \right ) \\
		&=c\left ( \frac{t}{\sqrt{n}} \right )\cdot \mathrm{exp}\left \{ i\frac{tA_{k}\cdot\sum_{k=1}^{h}m_{k}}{\sqrt{n}}-\frac{t^{2}\sigma^{2}_{k}\cdot\sum_{k=1}^{h}m_{k}}{2n} \right \}\\
		&\qquad\qquad\qquad\qquad\qquad\qquad\qquad\times \mathrm{exp}\left \{o\left ( \frac{t^{2}\sum_{k=1}^{h}m_{k}}{n} \right )    \right \} +d_{n}\left ( \frac{t}{\sqrt{n}} \right ),
	\end{align*}
	when $n$ goes to infinity,
	\begin{align*}
		E\left ( \mathrm{exp}\left \{ \frac{it}{\sqrt{n}}\left ( Q_{n}-\sum_{k=1}^{h}m_{k}A_{k} \right ) \right \}   \right )\to \mathrm{exp}\left \{-\frac{\mu^{2}t^{2}}{2}  \right \},\tag{3.1}
	\end{align*}
	which coincides with the characteristic function of the Gaussian distribution. Note that, due to the continuity of $c,\lambda $ and $d_{n}$ with respect to $t$, the left-hand side of (3.1) remains continuous within a certain neighborhood of the origin. Therefore, Theorem 2.1 implies that, provided  (iv) is fulfilled,  $\left ( Q_{n}-\sum_{k=1}^{h}m_{k}A_{k} \right ) /\sqrt{n}$ converges to a Gaussian distribution $\mathcal{N}\left ( 0,\mu^{2} \right ) $. As for the case of (v), for simplicity, we denote $\sum_{k=1}^{h}m_{k}\sigma_{k}^{2}/n$ as $x_{n}$, then by constructing a set $R$ such that
	\begin{align*}
		R=\left \{ x_{n}:x_{n} \le \eta^{2},\ \forall n \in \mathbb{N} \right \}, 
	\end{align*}
	we conclude from the Bolzano-Weierstrass theorem that there exists a convergent subsequence $\left \{ x_{n_{j}} \right \}_{j \in \mathbb{N}}$ of $\left \{x_{n}  \right \}_{n \in \mathbb{N}} $ and a positive constant $\zeta$ such that $	\lim_{j \to \infty}x_{n_{j}}=\zeta^{2}$. Meanwhile, we introduce a new set $R_{1}$ defined as
	\begin{align*}
		R_{1}=\left \{ x_{n_{j}}:	\lim_{j \to \infty}x_{n_{j}}=\zeta^{2},\ \forall n_{j} \in \mathbb{N} \right \}, 
	\end{align*}
	and consider the convergence of the Hausdorff distance $H\left ( R,R_{1} \right ) $ between $R$ and $R_{1}$,
	\begin{align*}
		\lim_{n \to \infty}H\left ( R,R_{1} \right )&=\lim_{n \to \infty}\max\left \{ h\left ( R,R_{1} \right ),h\left (R_{1},R  \right )   \right \}  \\
		&=\lim_{n \to \infty}\max\left \{  \max_{a \in R}\left \{\min_{b \in R_{1}}\left |a-b \right |   \right \} ,\max_{b \in R_{1}}\left \{ \min_{a \in R}\left|b-a  \right|  \right \}   \right \} =0,
	\end{align*}			
	so that for positive integers $n_{j}\in \mathbb{N}$ and $ h_{j}=h\left (  n_{j}\right )$ satisfying  $\sum_{k=1}^{h_{j}}m_{k}=n_{j}$, 
	\begin{align*}
		\lim_{n \to \infty}E\left ( \mathrm{exp}\left \{ \frac{it}{\sqrt{n_{j}}}\left ( Q_{n_{j}}-\sum_{k=1}^{h_{j}}m_{k}A_{k} \right ) \right \} \right ) = \mathrm{exp}\left \{ -\frac{\zeta^{2}t^{2}}{2} \right \}.
	\end{align*}
	Therefore, $\left ( Q_{n_{j}}-\sum_{k=1}^{h_{j} }m_{k} A_{k}\right ) /\sqrt{n_{j}}$ converges to a Gaussian distribution $\mathcal{N}\left ( 0,\zeta^{2} \right ) $. Similarly, it can be checked that
	\begin{align*}
		\lim_{n \to \infty}E\left (\mathrm{exp}\left \{ \frac{it}{n}\left ( Q_{n}-\sum_{k=1}^{h}m_{k}A_{k} \right )  \right \} \right )  =1\ \ \mathrm{and}\ \ 	\lim_{n_{j} \to \infty}E\left (\mathrm{exp}\left \{ \frac{it}{n_{j}}\left ( Q_{n}-\sum_{k=1}^{h_{j}}m_{k}A_{k} \right )  \right \} \right )  =1,
	\end{align*}
	that is, $\left ( Q_{n}-\sum_{k=1}^{h}m_{k}A_{k} \right ) /n $ and $\left ( Q_{n_{j}}-\sum_{k=1}^{h_{j}}m_{k}A_{k} \right ) /n$ both converges to 0 as $n$ goes to infinity.
	Our proof is complete.
\end{proof}

\begin{remark}
	Notice that there are no restrictions on the independence of the sequence $\left \{Y_{i}\right \}$. Therefore, the previous conclusions naturally hold when the random variables are mutually independent.
\end{remark}

\subsection{Higher-dimensional discrete time case}
Up to now, we have obtained results about the central limit theorem for sums of real random variables. In this paragraph, we extend our analysis to a more general situation, i.e., a sequence of $\mathbb{R}^{d}$-valued random variables $\left \{X_{i},i=1,\cdots,n  \right \} $ in $L^{2}$. We denote each variable $X_{i}$ as $\left ( X_{i1},\cdots,X_{id} \right )^{\mathrm{T}}\in \mathbb{R}^{d} $ and the sum of the first $n$ terms as $H_{n}$.

\setcounter{theorem}{0} 
\renewcommand{\thetheorem}{3.2}
\begin{theorem}
	Assume that there exist a neighbourhood $\mathcal{O}$ of $\mathbf{0}$ in $\mathbb{R}^{d}$ and functions $c\left ( t \right ) $, $\lambda_{k}\left ( t \right ) $ and $d_{n}\left ( t \right ) $ defined on $\mathcal{O} $, such that for all $t \in \mathcal{O}$ and $n \in \mathbb{N}$,
	\begin{align*}
		E\left (\mathrm{exp}\left \{ i\left \langle t,H_{n} \right \rangle\right \} \right )=\prod_{k=1}^{h }c\left ( t \right )\lambda_{k}\left ( t \right )^{m_{k}} +d_{n}\left ( t \right ) ,
	\end{align*}
	where $h=h\left ( n \right ) $ is a positive integer and $\left \{ m_{k} \right \}_{k=1,\cdots,h} $ satisfies $
	m_{1}+\cdots+m_{h}=n$. 
	Moreover, assume that
	
	\noindent (i) There exist $A_{k}$ and $\sigma_{k}$, $k=1,\cdots,h$, in $\mathbb{R}^{d}$, \color{black} such that when $t_{i} \to 0$ for $i=1,\cdots,d$,
	\begin{align*}
		\lambda_{k}\left ( t \right )=\mathrm{exp}\left \{ i\left \langle t,A_{k} \right \rangle -\frac{t^{\mathrm{T}}\sigma_{k}^{2}t}{2} +o\left ( t^{\mathrm{T}}t  \right )  \right \} .
	\end{align*}  
	(ii) The function $c$ is continuous at $\mathbf{0}$.
	
	\noindent (iii) The quantity $\left \| d_{n} \right \| _{L^{\infty}\left (\mathcal{O}\right )}$ tends to 0 when $n$ tends to infinity,
	
	\noindent and one of the following two additional assumptions is set up.
	
	\noindent (iv) Every element of the $d\times d$ matrix $\sum_{k=1}^{h}m_{k}\sigma_{k}^{2}/n$ has a non-negative limit as $n$ goes to infinity and all the corresponding limits form a new $d \times d$ matrix $\mu^{2}$.
	
	\noindent (v) Every element of the $d\times d$ matrix $\sum_{k=1}^{h}m_{k}\sigma_{k}^{2}/n$ has a non-negative bound, which leads to a new matrix $\eta^{2}$ consisted of all the corresponding bounds.
	
	\noindent Then when $n$ tends to infinity, under (i)-(iv), $\left (H_{n}-\sum_{k=1}^{h}m_{k}A_{k}  \right )/\sqrt{n}$ converges to a Gaussian distribution $\mathcal{N}\left (\mathbf{0},\mu^{2}  \right ) $,  while under (i)-(iii) and (v), there exist a $d \times d$ matrix $\zeta^{2} $, a sequence $\left \{ n_{j} \right \}_{j \in \mathbb{N}}\subset \left \{ n \right \}$ and a positive integer $h_{j}=h\left ( n_{j} \right ) $, such that every element of $\zeta^{2}$ is bounded by the corresponding element of $\eta^{2}$ and  $\left (H_{n_{j}}-\sum_{k=1}^{h_{j}}m_{k}A_{k}  \right )/\sqrt{n_{j}}$ converges to a Gaussian distribution $\mathcal{N}\left (\mathbf{0},\zeta^{2} \right ) $.
\end{theorem}
\begin{proof}
	First, taking $t = \mathbf{0}$, we obtain that
	$\lambda_{k}\left (\mathbf{0} \right ) =1$ for $k=1,\cdots,h$, and, correspondingly,
	\begin{align*}
		\lim_{n \to \infty}\prod_{k=1}^{h}c\left ( \mathbf{0} \right ) \lambda_{k}\left ( \mathbf{0} \right )^{m_{k}} =1-	\lim_{n \to \infty}d_{n}\left ( t \right )=1,
	\end{align*}  which indicates that $c\left (\mathbf{0} \right )=1 $. Next, similar to Theorem 3.1, 
	\begin{align*}
		&\lambda_{k}\left ( \frac{t}{\sqrt{n}} \right )=\mathrm{exp}\left \{ i\left \langle \frac{t}{\sqrt{n}},A_{k} \right \rangle -\frac{t^{\mathrm{T}}\sigma_{k}^{2}t}{2n}+o\left ( \frac{t^{\mathrm{T}}t }{n} \right )  \right \}  ,
	\end{align*}
	so that for $t \in \mathcal{O}$,
	\begin{align*}
		E\left (\mathrm{exp}\left \{ \frac{i\left \langle t,H_{n} \right \rangle}{\sqrt{n}} \right \} \right )&=\prod_{k=1}^{h}c\left ( \frac{t}{\sqrt{n}} \right )\lambda_{k}\left ( \frac{t}{\sqrt{n}} \right )^{m_{k}} +d_{n}\left ( \frac{t}{\sqrt{n}} \right )\\
		&=\prod_{k=1}^{h}c\left ( \frac{t}{\sqrt{n}} \right ) \mathrm{exp}\left \{ i\left \langle \frac{t}{\sqrt{n}},m_{k}A_{k} \right \rangle -\frac{t^{\mathrm{T}}m_{k}\sigma_{k}^{2}t}{2n}+o\left ( \frac{m_{k}t^{\mathrm{T}}t}{n} \right )  \right \} +d_{n}\left ( \frac{t}{\sqrt{n}} \right )\\
		&=c\left ( \frac{t}{\sqrt{n}} \right )\mathrm{exp}\left \{ i\left \langle \frac{t}{\sqrt{n}}, \sum_{k=1}^{h}m_{k}A_{k} \right \rangle -\frac{t^{\mathrm{T}} \left (  \sum_{k=1}^{h}m_{k}\sigma_{k}^{2} \right )t}{2n} \right \}\\
		&\qquad\qquad\qquad\qquad\qquad\qquad\qquad\times \mathrm{exp}\left \{ o\left ( \frac{\sum_{k=1}^{h}m_{k}t^{\mathrm{T}}t}{n} \right ) \right \}   +d_{n}\left ( \frac{t}{\sqrt{n}} \right ).
	\end{align*}
	This implies that when $n$ goes to infinity,
	\begin{align*}
		E\left ( \mathrm{exp}\left \{ \frac{i\left \langle t, H_{n}-\sum_{k=1}^{h}m_{k}A_{k} \right \rangle }{\sqrt{n}} \right \} \right ) \to  \mathrm{exp}\left \{ \frac{-t^{\mathrm{T}}\mu^{2} t}{2} \right \}.
	\end{align*}
 Using again Theorem 2.1, if (iv) is satisfied, then $\left ( H_{n}-\sum_{k=1}^{h}m_{k}A_{k} \right ) /\sqrt{n}$ converges to a Gaussian distribution $\mathcal{N}\left ( \mathbf{0},\mu^{2} \right ) $.  If (v) is satisfied, then 
	we can find a sequence $\left \{ n_{j} \right \}_{j \in \mathbb{N}}\subset \left \{ n \right \}$, a positive integer $h_{j}=h\left ( n_{j} \right ) $ and a $d \times
	d$ matrix $\zeta^{2}$, such that $\left (H_{n_{j}}-\sum_{k=1}^{h_{j}}m_{k}A_{k}  \right )/\sqrt{n_{j}}$ converges to a Gaussian distribution $\mathcal{N}\left (\mathbf{0},\zeta^{2} \right ) $,  where every element of $\zeta^{2}$ is bounded by the corresponding element of $\eta^{2}$. Finally, the convergence
	\begin{align*}
		\lim_{n \to \infty}E\left ( \mathrm{exp}\left \{ \frac{i\left \langle t,  H_{n}-\sum_{k=1}^{h}m_{k}A_{k}\right \rangle}{n} \right \} \right )=\lim_{n_{j} \to \infty}E\left ( \mathrm{exp}\left \{ \frac{i\left \langle t,  H_{n_{j}}-\sum_{k=1}^{h_{j}}m_{k}A_{k}\right \rangle}{n_{j}} \right \} \right )=1 
	\end{align*}
	can be also established.
\end{proof}

\subsection{Continuous time case}
Now we turn to study the corresponding case with continuous time. Specifically, we consider an $\mathbb{R}^{d}$-valued stochastic \color{black} process $\left \{ x\left (s  \right ) \right \}_{s \ge 0}$ in $L^{2}$ and the associated integral $	\mathcal{S}\left ( T \right ) =\int_{0}^{T}x\left (s  \right ) ds,T\ge0 $.
\setcounter{theorem}{0} 
\renewcommand{\thetheorem}{3.3}
\begin{theorem}
	Assume that there exist a neighbourhood $\mathcal{O}$ of $\mathbf{0}$ in $\mathbb{R}^{d}$ and functions $c\left ( t \right ) $, $\lambda_{k}\left ( t \right ) $ and $d_{T}\left ( t \right ) $ defined on $\mathcal{O}$, such that for all $t \in \mathcal{O}$ and $T\ge 0$, 
	\begin{align*}
		E\left (\mathrm{exp}\left \{ i\left \langle t,\mathcal{S}_{T} \right \rangle \right \} \right )=c\left ( t \right )\mathrm{exp}\left \{\int_{1}^{h}m_{s}\ln\lambda_{s}\left ( t \right ) ds  \right \} +d_{T}\left ( t \right ) ,
	\end{align*}
	where $h=h\left ( T \right ) $ is a positive integer and $\left \{ m_{k} \right \}_{k \in \left [1,h\right]   }$ satisfies $\int_{1}^{h}m_{k}dk=T$.
	Moreover, assume that:
	
	\noindent (i) There exist $A_{s}$ and $\sigma_{s}$, $s \in \left [1,h \right ] $, in $\mathbb{R}^{d}$, \color{black} such that when $t_{i} \to 0$ for $i=1,\cdots,d $,
	\begin{align*}
		\lambda_{s}\left (t \right )=\mathrm{exp}\left \{ i\left \langle t, A_{s}\right \rangle -\frac{t^{\mathrm{T}}\sigma_{s}^{2}t}{2}+o\left ( t^{\mathrm{T}}t \right )  \right \} .
	\end{align*}  
	
	\noindent (ii) The function $c$ is continuous at $\mathbf{0}$.
	
	\noindent (iii) The quantity $\left \| d_{T} \right \| _{L^{\infty}\left (\mathcal{O}\right ) }$ tends to 0 when $T$ tends to infinity,
	
	\noindent and one of the following two additional assumptions is set up.

	\noindent (iv) Every element of the $d\times d$ matrix $\int_{1}^{h}m_{s}\sigma^{2}_{s}ds/T$ has a non-negative limit as $T$ goes to infinity and all the corresponding limits form a new $d \times d$ matrix $\mu^{2}$.
	
	\noindent (v) Every element of the $d\times d$ matrix $\int_{1}^{h}m_{s}\sigma^{2}_{s}ds/T$  has a non-negative bound, which leads to a new matrix $\eta^{2}$ consisted of all the corresponding bounds.
	
	\noindent Then when $T$ tends to infinity, under (i)-(iv), $\left ( \mathcal{S}_{T}-\int_{1}^{h}m_{s}A_{s}ds \right ) /\sqrt{T}$ converges to a Gaussian distribution $\mathcal{N}\left (\mathbf{0},\mu^{2}  \right ) $,  while under (i)-(iii) and (v),
	there exists a $d \times d$ matrix $\zeta^{2} $, a sequence $\left \{ T_{j} \right \}_{j \in \mathbb{N} } \subset \left \{ T\right \}$ and a positive integer $h_{j}=h\left ( T_{j} \right ) $, such that every element of $\zeta^{2}$ is bounded by the corresponding element of $\eta^{2}$ and $\left ( \mathcal{S}_{T_{j}}-\int_{1}^{h_{j}}m_{s}A_{s}ds \right ) /\sqrt{T_{j}}$ converges to a Gaussian distribution $\mathcal{N}\left (\mathbf{0},\zeta^{2} \right ) $.
\end{theorem}
\begin{proof}
	First, taking $t = \mathbf{0}$, we obtain that $c\left ( \mathbf{0} \right ) =1-d_{T}\left ( \mathbf{0} \right )\to 1 $ as $\lambda_{s}\left (\mathbf{0} \right ) =1$ for $s \in \left [1,h \right ] $, then $c\left ( \mathbf{0} \right )=1 $. Next, we use the asymptotic expansion to $\lambda_{s}\left ( t \right ) $ and obtain 
	\begin{align*}
		E\left (\mathrm{exp}\left \{ \frac{i\left \langle t,\mathcal{S}_{T} \right \rangle}{\sqrt{T}} \right \} \right )&=c\left ( \frac{t}{\sqrt{T}} \right )\mathrm{exp}\left \{ \int_{1}^{h}m_{s}\ln \lambda_{s}\left (  \frac{t}{\sqrt{T}}\right )ds\right \}  +d_{T}\left ( \frac{t}{\sqrt{T}} \right ) \\
		&=c\left ( \frac{t}{\sqrt{T}} \right )\mathrm{exp}\left \{ i\left \langle \frac{t}{\sqrt{T}},\int_{1}^{h}m_{s}A_{s}ds \right \rangle -\frac{t^{\mathrm{T}}\left ( \int_{1}^{h}m_{s}\sigma_{s}^{2}ds \right ) t}{2T}\right \}\\
		&\qquad\qquad\qquad\qquad\qquad\qquad\qquad \times\mathrm{exp}\left\{ o\left (\frac{ t^{\mathrm{T}}t\int_{1}^{h}m_{s}ds }{T}\right )   \right \}     +d_{T}\left ( \frac{t}{\sqrt{T}} \right ) .
	\end{align*}
	Therefore, when $T$ goes to infinity,
	\begin{align*}
		E\left ( \mathrm{exp}\left \{ \frac{i\left \langle t, \mathcal{S}_{T}-\int_{1}^{h}m_{s}A_{s}ds \right \rangle}{\sqrt{T}} \right \} \right )\to \mathrm{exp}\left \{ -\frac{t^{\mathrm{T}}\mu^{2}t}{2} \right \}.
	\end{align*}
	Under (i)-(iii), if (iv) is moreover satisfied, then $\left ( \mathcal{S}_{T}-\int_{1}^{h}m_{s}A_{s}ds\right )/\sqrt{T}$ converges to a Gaussian distribution $\mathcal{N}\left ( \mathbf{0},\mu^{2} \right ) $. If we substitute (iv) with (v), then by finding a sequence $\left \{ T_{j} \right \}_{j\in \mathbb{N}}\subset \left \{ T \right \}  $, a positive integer $h_{j}=h\left (T_{j}  \right ) $ and a matrix $\zeta^{2}$,
	such that every element of $\int_{0}^{h_{j}}m_{k}\sigma^{2}_{k}dk/T_{j}$ converges to the corresponding element of  $\zeta^{2}$ and is bounded by the corresponding element of $\eta^{2}$, we conclude that $\left ( \mathcal{S}_{T_{j}}-\int_{1}^{h_{j}}m_{s}A_{s}ds\right )/\sqrt{T_{j}}$ converges to a Gaussian distribution $\mathcal{N}\left ( \mathbf{0},\zeta^{2} \right ) $.	Finally, it can be proved that
	\begin{align*}
		E\left ( \mathrm{exp}\left \{ \frac{i\left \langle t, \mathcal{S}_{T}-\int_{1}^{h}m_{s}A_{s}ds \right \rangle}{T} \right \} \right )\to 1\ \ \mathrm{and}\ \ E\left ( \mathrm{exp}\left \{ \frac{i\left \langle t, \mathcal{S}_{T_{j}}-\int_{1}^{h_{j}}m_{s}A_{s}ds \right \rangle}{T_{j}} \right \} \right )\to 1,
	\end{align*}
	as $T$ goes to infinity.
\end{proof}

\section{Applications}
In this section, we focus on studying the integrable Hamiltonian systems with stochastic perturbations, 
in which the noise is Gaussian or non-Gaussian. For the sake of convenience, we'll consistently express the process $x\left ( T \right ) $ in the form of $x_{T}$ for time $T \ge 0$. 

\subsection{One-dimensional Integrable Hamiltonian systems with white noise perturbations}
First of all, we consider the following nonlinear equation
\begin{align*}
	\ddot{x}+x^{2m+1}=0,\ \ m \ge 1.\tag{4.1}
\end{align*}
Then the corresponding Hamiltonian system can be written in the following form
\begin{align*}
	\dot{x}=\frac{\partial}{\partial y}h\left ( x,y \right ) ,\ \ \ \ \dot{y}=-\frac{\partial}{\partial x}h\left ( x,y \right ), \tag{4.2}
\end{align*}
and $h\left ( x,y \right )  =y^{2}/2+x^{2m+2}/\left ( 2m+2 \right ) $. It can be seen that $h>0$ on $\mathbb{R}^{2}$ except at the only equilibrium point $\left ( x,y \right )=\left ( 0,0 \right )  $. Referring to the ideas in \cite{12} and \cite{22}, all the solutions of (4.2) are  periodic while the periods tend to zero as $h$ tend to infinity. Moreover, if we denote by $T^{\ast}>0$ its minimal period and $\left ( C\left ( t \right ),S\left ( t \right )   \right ) $ the solution of (4.2) satisfying the initial condition $\left ( C\left ( 0 \right ),S\left ( 0 \right )   \right )  = \left ( 1,0 \right ) $, then we may 
define the action-angle variables by the map $\psi:\mathbb{R}^{+}\times \mathcal{S}^{1} \to \mathbb{R}^{2}/\left \{ 0 \right \} $ where $\mathcal{S}^{1}$ is the unit circle, $\left ( x,y \right )=\psi\left ( I,\theta \right )  $ for $I>0$ and $\theta \left (\mathrm{mod}\ 1  \right ) $. To be specific,
\begin{align*}
	\psi:\ \ 	x=c^{\alpha}I^{\alpha}C\left ( \theta T^{\ast} \right ) ,\ \ \ y=c^{\beta}I^{\beta}S\left ( \theta T^{\ast} \right ), 
\end{align*}
where $\alpha=1/m+2$, $\beta=1-\alpha$ and $c=1/T^{\ast}\alpha$. Thus, the new Hamiltonian function is defined as
\begin{align*}
	h\circ \psi\left ( I,\theta \right )=\frac{c^{2\beta}}{2m+2} \cdot I^{2\beta}=\tilde{h}\left ( I\right ) ,
\end{align*} 
and (4.1) turns to
\begin{align*}
	\dot{I}=-\frac{\partial \tilde{h}}{\partial \theta}=0,\ \ \dot{\theta}=\frac{\partial \tilde{h}}{\partial I}=\frac{\beta c^{2\beta}}{m+1} \cdot I^{2\beta-1}\triangleq \omega\left ( I \right ) .\tag{4.3}
\end{align*}  

Consequently, this becomes an integrable system in the coordinates of the action-angle variable $\left (  I,\theta\right ) $. In light of this, we consider adding stochastic perturbations $\xi^{i}$ for $i=1,2$ to (4.3), which, as we mentioned earlier, is of great practical significance for the original model. Specifically, the noise here belongs to Brownian motions or $\mathrm{L\acute{e}vy}$ processes.
Therefore, it's appropriate to focus on the new system
\begin{align*}
	\dot{I}_{s}=\dot{\xi}_{s}^{1},\ \ \ \ \dot{\theta}_{s}=\omega\left (  I_{s}\right ) +\dot{\xi}^{2}_{s}.\tag{4.4}
\end{align*}  
We denote by $I_{0}$ and $\theta_{0}$ the initial values of $I_{s}$ and $\theta_{s}$, respectively, and we consider it as a general situation. We will see below that the orbits of systems (4.4) follow an asymptotic normal distribution. 

\subsubsection{Gaussian noise perturbations}
For Gaussian noise, we write (4.4) as follows
\begin{align*}
	I_{s}=I_{0}+ \sigma B^{1}_{s},\ \ \ \ 
	\theta_{s}= \theta_{0}+\int_{0}^{s}\omega \left ( I_{r} \right ) dr+\zeta B^{2}_{s},\ \ s \ge 0,
\end{align*}
where $\sigma$ and $\zeta$ are two constants, $B^{1}$ and $B^{2}$ are two one-dimensional independent standard Brownian motions defined on a filtered probability space $\left ( \Omega,\mathcal{F},\mathcal{F}_{t} ,\mathbb{P} \right ) $ satisfying the standard conditions. We denote by $\mathbb{E}$ the corresponding expectation of $\mathbb{P}$.

\setcounter{theorem}{0} 
\renewcommand{\thetheorem}{4.1}
\begin{theorem}
	There exist $\varepsilon,\delta>0$, $A_{k}\in\mathbb{R}$ and a positive integer $h=h\left ( n \right ) $,
	such that 
	\begin{align*}
		\frac{1}{\sqrt{n\delta}}\left (   \sum_{k=1}^{n}\frac{\theta_{k\delta}}{k\delta}-\sum_{k=1}^{h }m_{k\delta}A_{k\delta}\right ) 
	\end{align*} 
	converges to a Gaussian distribution $\mathcal{N}\left (0,\left ( \int_{\Omega} \omega\left (x  \right )dx  \right )^{2}  \right ) $, where $\sum_{k=1}^{h }m_{k\delta}=n\delta$.
\end{theorem}
\begin{proof}
	Our main idea is to construct conditions of Theorem 3.1, so that to take advantage of its conclusion.
	We choose to denote $T$ as $n\delta$, where $\delta$ is some positive constant, then we divide $\left [0,T  \right ] $ into $n$ parts with each length $\delta$, so  that when $n$ goes to infinity, $T$ goes to infinity as well.
	For simplicity, we'll take $\delta=1$ firstly and define sequences $\left \{ X_{i} \right \}_{i=1,\cdots,n}$ and $\left \{ P_{i} \right \}_{i=1,\cdots,n}$ as follows,
	\begin{align*}
		&X_{1}=\left ( \theta_{0}+\int_{0}^{1}  \omega\left ( I_{0}+\sigma B_{s}^{1} \right )ds \right ) ,\ \ P_{1}=\eta B_{1}^{2},\\
		&X_{2}=\frac{1}{2}\left ( \theta_{0}+\int_{0}^{2}  \omega\left ( I_{0}+\sigma B_{s}^{1} \right )ds \right ),\ \ P_{2}=\frac{\eta B_{2}^{2}}{2},  \\
		&\cdots\\
		&X_{n} =\frac{1}{n }\left ( \theta_{0}+\int_{0}^{n}  \omega\left ( I_{0}+\sigma B_{s}^{1} \right )ds \right ),\ \ P_{n}=\frac{\eta B_{n}^{2}}{n},
	\end{align*}
	together with
	\begin{align*}
		D_{n}=\sum_{k=1}^{n}\frac{\theta_{k}}{k}=\sum_{k=1}^{n}X_{k}+\sum_{k=1}^{n}P_{k}\triangleq H_{n}+V_{n}.		\end{align*}
	Based on the independence between different $\theta_{k}$ and $\theta_{j}$, for $\varepsilon>0$, 
	\begin{align*}
		\mathbb{E}\left ( \mathrm{exp}\left \{ i tD_{n} \right \} \right ) &=\mathbb{E}\left ( \mathrm{exp}\left \{ i t\sum_{k=1}^{n}\frac{\theta_{k}}{k} \right \} \right )=\prod_{k=1}^{n}\mathbb{E}\left ( \mathrm{exp}\left \{\frac{ i t\theta_{k}}{k}  \right \} \right ) ,\ \ t\in \left [-\varepsilon,\varepsilon  \right ] .\tag{4.5}
	\end{align*}
	By picking an integer sequence $\left \{m_{k}  \right \}_{k=1,\cdots,h} $ satisfying $\sum_{k=1}^{h}m_{k}=n$, where $h=h\left ( n \right ) $ is a positive integer, we can classify the random variable sequence $\left \{ \theta_{j}/j \right \}_{j=1,\cdots,n}$ into $h$ different groups according to the expressions of characteristic functions and represent the characteristic function of the $k$-th \textquotedblleft group\textquotedblright as $\lambda_{k}\left ( t \right )$ for $k=1,\cdots,h$, with $m_{k}$ numbers of random variables in the sequence $\left \{  X_{j}+P_{j}  \right \}_{j=1,\cdots,n}$. For simplicity, we denote the variables in $h$ groups as $Z_{1},\cdots,Z_{h}$. Apparently, they are independent of each other. Hence, 
	(4.5) equivalently turns to
	\begin{align*}
		\mathbb{E}\left ( \mathrm{exp}\left \{ i tD_{n}  \right \} \right ) =\prod_{k=1}^{h}c\left ( t \right ) \lambda_{k}\left (t  \right )^{m_{k}}, 
	\end{align*}
	where $c\left ( t \right )=1$, $\lambda_{k}\left ( t \right ) =\mathbb{E}\left (   \mathrm{exp}\left \{ i tZ_{k}  \right \}\right )$
	has an approximation
	\begin{align*}
		W_{k}
		&=1+it\mathbb{E}Z_{k}-\frac{t^{2}\mathbb{E} Z_{k}^{2}}{2} +o\left ( t^{2} \right ),\ \ t \in \left [-\varepsilon,\varepsilon  \right ], 
	\end{align*}
	that is, $\lambda_{k}\left ( t \right ) $ can be seen as
	\begin{align*}
		\lambda_{k}\left ( t \right )  =\mathrm{exp}\left \{ i t\mathbb{E}Z_{k} -\frac{t^{2}\mathbb{E} Z_{k}^{2} }{2}+o\left ( t^{2} \right )  \right \} ,\ \  t \in \left [-\varepsilon,\varepsilon  \right ].
	\end{align*}
	If we denote $A_{k}=\mathbb{E}Z_{k}$ and $ \sigma_{k}^{2}=\mathbb{E}Z_{k}^{2}$, then 
	we've successfully proved (i), (ii) and (iii) of Theorem 3.1: $c\left ( t \right ) =1$ is continuous at $t=0$, $\left \| d_{n}\left ( t \right ) \right \|_{L^{\infty}\left [-\varepsilon,\varepsilon \right ] } =0$ has the null limit as $n$ goes to infinity. Therefore, it's sufficient to check the feasibility of either (iv) or (v) of Theorem 3.1.  We'll discuss each variable $\theta_{k}/k$  for fixed $k$ firstly. Under the assumptions of independence,
	$\mathbb{E}\left ( P_{k}X_{k} \right )=\mathbb{E}P_{k}\cdot\mathbb{E}X_{k}=0$,
	so that
	\begin{align*}
		\frac{1}{n}\sum_{k=1}^{h}m_{k}\sigma_{k}^{2}&=\frac{1}{n}\sum_{k=1}^{n}	\mathbb{E}P_{k}^{2}+\frac{1}{n}\sum_{k=1}^{n}	\mathbb{E}X_{k}^{2}+\frac{2}{n}\sum_{k=1}^{n}	\mathbb{E}\left ( P_{k}X_{k} \right ) \\
		&= \frac{1}{n}\sum_{k=1}^{n}\left [ \zeta^{2}\cdot  \mathbb{E}\left (\frac{B_{k}^{2}}{k} \right )^{2} +  \mathbb{E}\left ( \frac{\theta_{0}}{k}+\frac{1}{k}\int_{0}^{k}\omega\left ( I_{0} +\sigma B_{r}^{1}\right )dr \right )^{2}  \right ]\\
		&= \frac{1}{n}\sum_{k=1}^{n}\frac{\zeta^{2}}{k}  + \frac{1}{n}\sum_{k=1}^{n} \mathbb{E}\left ( \frac{\theta_{0}}{k}+\frac{1}{k}\int_{0}^{k}\omega\left ( I_{0} +\sigma B_{r}^{1}\right )dr \right )^{2}  .
	\end{align*}
	According to the Birkhoff ergodic theorem, 
	\begin{align*}
		\frac{1}{k}\int_{0}^{k}\omega\left ( I_{0}+\sigma B_{s}^{1} \right )ds  = \int_{\Omega}\color{black}\omega\left (x  \right )dx.\tag{4.6}
	\end{align*}
	If we denote 
	\begin{align*}
		z_{k}=\left (\frac{1}{k}\int_{0}^{k}\omega\left ( I_{0}+\sigma B_{s}^{1} \right )ds  \right )^{2} -\left ( \int_{\Omega}\omega\left (x  \right )dx \right )^{2} ,
	\end{align*} 
	then with (4.6), for any $\xi>0$, there exists an integer $N_{1}\in \mathbb{N}$, such that for $k>N_{1}$ and $	\left |z_{k} \right |<\xi$,
	\begin{align*}
		\lim_{n \to \infty}\frac{1}{n} \sum_{k=1}^{n}z_{k}&=	\lim_{n \to \infty}\frac{1}{n} \sum_{k=1}^{N_{1}}z_{k}+	\lim_{n \to \infty}\frac{1}{n} \sum_{k=N_{1}+1}^{n}z_{k}=0.\tag{4.7}
	\end{align*}
	Moreover, since 
	\begin{align*}
		\lim_{k \to \infty}\mathbb{E}\left (\frac{\theta_{k}}{k}  \right )^{2}=\lim_{k \to \infty}\mathbb{E}\left (  \frac{\theta_{0}}{k}   \right )\cdot \mathbb{E}\left (  \frac{1}{k}\int_{0}^{k}  \omega\left ( I_{0}+\sigma B_{s}^{1} \right )ds   \right )=0,
	\end{align*}
	we obtain 
	\begin{align*}
		\lim_{n \to \infty}\frac{1}{n}\sum_{k=1}^{n}\mathbb{E}X_{k}^{2}&=\lim_{n \to \infty}\frac{1}{n}\sum_{k=1}^{n}\mathbb{E}\left (  \frac{\theta_{0}}{k}+\frac{1}{k}\int_{0}^{k}  \omega\left ( I_{0}+\sigma B_{s}^{1} \right )ds   \right )^{2}=\left ( \int_{\Omega} \omega\left ( x\right )dx   \right )^{2}.
	\end{align*}
	It then follows that
	\begin{align*}
		\lim_{n \to \infty}\frac{1}{n}\sum_{k=1}^{h}m_{k}\sigma_{k}^{2}&=\lim_{n \to \infty}\frac{\zeta^{2}O\left ( \ln n \right ) }{n}+\left ( \int_{\Omega} \omega\left ( x\right )dx   \right )^{2}=\left ( \int_{\Omega} \omega\left ( x\right )dx   \right )^{2},
	\end{align*}
	which corresponds to (iv) of Theorem 3.1. Thus, we conclude our proof.
\end{proof}
\subsubsection{Non-Gaussian noise perturbations}
In this paragraph, we'll delve into the system (4.4) with L$\mathrm{\acute{e}vy}$ noise, which is appropriate for the study of non-Gaussian perturbations with jumps.

Consider the integrable system
\begin{align*}
	I_{s}=I_{0}+\sigma  L_{s}^{1} ,\ \ \ \ 
	\theta_{s} =\theta_{0}+ \int_{0}^{s}\omega\left ( I_{r} \right )  dr+\zeta L^{2}_{s}, \ \ s \ge 0, \tag{4.8}
\end{align*}
where $L^{1}$ and $L^{2}$ are two independent one-dimensional $\mathrm{L\acute{e}vy}$ noise corresponding to $\left ( \gamma_{1},\xi_{1},N_{1},\nu_{1} \right ) $ and $\left (  \gamma_{2},\xi_{2},N_{2},\nu_{2}\right ) $,   $\sigma$ and $\zeta$ are two constants. Specifically, (4.8) can be written in the following form,
\begin{align*}
	\left\{
	\begin{aligned}
		I_{s}&=I_{0}+\sigma \gamma_{1} s+\sigma \xi_{1}\color{black}B_{s}^{1}+\sigma\int_{\left |z  \right |<1 } z\tilde{N}_{1}\left ( s,dz \right ) +\sigma\int_{\left |z  \right |\ge 1 } zN_{1}\left ( s,dz \right ) ,\\
		\theta_{s} &= \theta_{0}+\int_{0}^{s}\omega\left ( I_{r} \right ) dr+\zeta\gamma_{2} s +\zeta\xi_{2} B_{s}^{2}+\zeta\int_{\left |z  \right |<1 }z\tilde{N}_{2}\left ( s,dz \right )+\zeta\int_{\left |z  \right |\ge 1 }zN_{2}\left ( s,dz \right )  ,
	\end{aligned}
	\right.
\end{align*}
where $\int\min\left \{\left |z  \right |^{2} ,1  \right \} \nu_{i}\left ( dz \right )  <\infty$ for $i=1,2$, the assumptions related to Brownian motions and their corresponding probability spaces remain identical to those outlined in Theorem 4.1. Referring to \cite{01}, we employ the interlacing technique, which allows us to disregard jumps of $\theta$ that exceed a size of 1 and instead focus on the subsequent equation for $\theta_{s}$,
\begin{align*}
	\theta_{s} =\theta_{0}+\int_{0}^{s}\omega\left ( I_{r} \right ) dr+\zeta\gamma_{2} s +\zeta \xi_{2}B_{s}^{2}+\zeta\int_{\left |z  \right |<1 }z\tilde{N}_{2}\left ( s,dz \right ).
\end{align*}

\setcounter{theorem}{0} 
\renewcommand{\thetheorem}{4.2}
\begin{theorem}
	There exist $\varepsilon,\delta>0$ and $A_{k},\mu \in \mathbb{R}$
	such that $\left ( \sum_{k=1}^{n}\theta_{k\delta}/k\delta -\sum_{k=1}^{h }m_{k\delta}A_{k\delta}\right ) /\sqrt{n\delta}    $ converges to a Gaussian distribution $\mathcal{N}\left (0,\mu^{2} \right ) $, where $h=h\left ( n \right )$ is a positive integer, $\sum_{k=1}^{h}m_{k\delta}=n\delta$, $\mu^{2}$ is bounded below by $\left ( \int_{\Omega}\omega\left ( x \right ) dx \right ) ^{2}$ and bounded above by $2\zeta^{2}\gamma_{2}^{2}+\left ( \int_{\Omega}\omega\left ( x \right ) dx \right ) ^{2}$.
\end{theorem}
\begin{proof}
	Firstly, we denote $T$ as $n\delta$ for a positive constant $\delta$ (for simplicity we suppose $\delta=1$).
	Next, we define  $D_{n}=\sum_{k=1}^{n}\theta_{k}/k=\sum_{k=1}^{n}\left ( X_{k}+P_{k} \right )=H_{n}+V_{n}$,
	where
	\begin{align*}
		&X_{k}=\frac{1}{k}\left ( \theta_{0}+\int_{0}^{k}\omega\left ( I_{r} \right ) dr +\zeta\xi_{2}B_{k}^{2}\right ),\ \  H_{n}=\sum_{k=1}^{n}X_{k}, \\ 
		&P_{k}=\frac{\zeta}{k} \left (k\gamma_{2} +\int_{\left |z  \right |<1 }z \tilde{N}_{2}\left ( k,dz \right )   \right ), \ \  V_{n}=\sum_{k=1}^{n}P_{k}.
	\end{align*} 
In accordance with the previous demonstrations, the characteristic function of $D_{n}$ can be expressed as 
\begin{align*}
	\mathbb{E}\left (  \mathrm{exp}\left \{i tD_{n}  \right \} \right )&= \prod_{k=1}^{n}\mathbb{E}\left ( \mathrm{exp}\left \{it\left ( X_{k}+P_{k} \right )   \right \}\right ) \triangleq \prod_{k=1}^{h}c\left ( t \right ) \lambda_{k}\left ( t \right ) ^{m_{k}}+d_{n}\left ( t \right ) ,
\end{align*}
where $c\left ( t \right ) =1$, $d_{n}\left (  t\right )=0 $, $h=h\left (n  \right ) $ is a positive integer, the sequence $\left \{ m_{k} \right \}_{m=1,\cdots,h} $ satisfies $\sum_{k=1}^{h}m_{k}=n$, $\left \{ \lambda_{k} \right \}_{k=1,\cdots,h}$ and $\left \{ Z_{k} \right \}_{k=1,\cdots,h}$ are the shorthands, which describe all the existing different characteristic functions and the corresponding random variables of sequence $\left \{ X_{j}+P_{j} \right \}_{j=1,\cdots,n}$. To be specific, for $\varepsilon>0$, $\lambda_{k}\left ( t \right ) $ is defined as
\begin{align*}
	\lambda_{k}\left ( t \right ) =\mathbb{E}\left ( \mathrm{exp}\left \{ itZ_{k} \right \}  \right )
	=\mathrm{exp}\left \{ itA_{k} -\frac{t^{2}\sigma_{k}^{2}}{2} +o\left (t^{2} \right )  \right \} ,\ \ t \in \left [-\varepsilon,\varepsilon  \right ] ,
\end{align*}
where $A_{k}=\mathbb{E}Z_{k}$ and $\sigma_{k}^{2}=\mathbb{E}Z_{k}^{2}$. Therefore, we've successfully constructed (i),  (ii) and (iii) of Theorem 3.1. Now we turn to prove either (iv) or (v) of Theorem 3.1.  Equivalent to Theorem 4.1, $\lim_{n \to \infty}\sum_{k=1}^{n}\mathbb{E}X_{k}^{2}/n$ is well defined  and can be expressed as $\left ( \int_{\Omega}\omega\left ( x \right ) dx \right ) ^{2}$. Meanwhile, 
\begin{align*}
	\mathbb{E}P_{k}^{2}\color{black}&=\mathbb{E}\left [    \zeta\gamma_{2}+\frac{\zeta}{k}\int_{\left |z_{i}  \right |<1 }z\tilde{N}_{2}\left ( k,dz \right )  \right ]^{2}
	\le 2\zeta^{2}\gamma_{2}^{2}+ \frac{2\zeta^{2}}{k}\int_{\left |z  \right |<1 }\left |z  \right |^{2}\nu\left ( dz \right )  ,
\end{align*}
where we use a property of $\mathrm{L\acute{e}vy}$ processes,
\begin{align*}
	\mathbb{E}\left (  \int_{\left |z  \right |<1 }z\tilde{N}_{2}\left ( k,dz \right )  \right )^{2}=\int_{\left |z \right |<1}k\left |z \right |^{2}\nu\left ( dz \right ).
\end{align*}
Therefore, taking $n$ to infinity, we have
\begin{align*}
	\lim_{n \to \infty}\frac{1}{n}\sum_{k=1}^{n}\mathbb{E}P_{k}^{2}&\le 	2\zeta^{2}\gamma_{2}^{2}+\lim_{n \to \infty}\frac{2\mathcal{O}\left (\ln n  \right ) }{n}\cdot\int_{\left |z  \right |<1 }\left |z  \right |^{2}\nu\left ( dz \right )=2\zeta^{2}\gamma^{2}_{2} .
\end{align*}
To this end, 
\begin{align*}
	\lim_{n \to \infty}\frac{1}{n}\sum_{k=1}^{h}m_{k}\sigma_{k}^{2}&=\lim_{n \to \infty}\frac{1}{n}\sum_{k=1}^{n}\mathbb{E} X_{k}^{2} + \lim_{n \to \infty}\frac{1}{n}\sum_{k=1}^{n}\mathbb{E} P_{k}^{2}\le 2\zeta^{2}\gamma_{2}^{2}+\left ( \int_{\Omega}\omega\left ( x \right ) dx \right ) ^{2},
\end{align*}
while by Theorem 4.1, 
\begin{align*}
	\lim_{n \to \infty}\frac{1}{n}\sum_{k=1}^{h}m_{k}\sigma_{k}^{2} \ge \left ( \int_{\Omega}\omega\left ( x \right ) dx \right ) ^{2}.
\end{align*}
This precisely aligns with (iv) of Theorem 3.1, thereby bringing our proof to the conclusion.
\end{proof}

Next, we consider a simple pendulum system. Suppose $l$ and $\theta$ are the length and angle of the single pendulum respectively, then the Lagrangian of such system can be denote as
\begin{align*}
L=T-V=\frac{ml^{2}\dot{\theta}^{2}}{2}-mgl\left ( 1-\cos\theta \right ) ,
\end{align*}
where $T=ml^{2}\dot{\theta} ^{2}/2$ is the kinetic energy and $V=mgl\left ( 1-\cos \theta\right )$ is the potential energy, $m$ is the mass of the pendulum and $g$ is the acceleration of gravity. 
Applying the Euler-Lagrange equation, we obtain
\begin{align*}
\frac{d}{dt}\left ( \frac{\partial L}{\partial \dot{\theta}} \right ) -\frac{\partial L}{\partial \theta}=ml^{2}\ddot{\theta}-mgl\sin\theta=0,\tag{4.9}
\end{align*}
which indicates that the pendulum equation can be simplified as $\ddot{\theta}+\sin\theta=0$. If we denote $x=\dot{\theta}$, then (4.9) equivalently turns to
\begin{align*}
\dot{\theta}=x,\ \ \dot{x}=-\sin\theta.
\end{align*}
By calculating $\dot{x}=\dot{\theta}=0$, we obtain the set $M$ of equilibrium points, that is,
\begin{align*}
M=\left \{ \theta:\theta=n\pi,\ n \in \mathbb{Z}\right \}.
\end{align*}
To be specific, on the one hand, if $n=2k\pi, k \in \mathbb{Z}$, that is, $\left ( \theta,x \right )=\left ( 2k\pi,0 \right )  $. Through linearization, we obtain that the eigenvalues of the coefficient matrix are all pure imaginary numbers, so that $\theta=2k\pi,k \in \mathbb{Z}$ are the centered equilibrium points of the linearized system, that is, the original single pendulum will acts as a periodic and closed orbit near $\theta=2k\pi,k \in \mathbb{Z}$. On the other hand, for $n=\left (2k+1  \right )\pi ,k \in \mathbb{Z}$, that is, $\left ( \theta,x \right )=\left ( \left ( 2k+1 \right ) \pi,0 \right )  $, in the same way, it can be proved that the eigenvalues of the corresponding coefficient matrix have positive or negative real parts, thereby indicating that such points do not constitute the stable equilibrium points that we seek. Consequently, we shall pay our attention on the dynamics of the pendulum 
when $\theta$ approaches to $2k\pi,k\in \mathbb
Z$.

If we introduce the momentum as $p_{\theta}=\partial L/\partial \dot{\theta}=ml^{2}\dot{\theta}^{2}$,
then according to the Legendre transformation, the corresponding Hamiltonian $H\left (p_{\theta},\theta \right ) $ can be defined as follows
\begin{align*}
H=T+V \approx\frac{p_{\theta}^{2}}{2ml^{2}}+\frac{mgl\theta^{2}}{2},\ \ \mathrm{where}\ \left |\theta-2k\pi  \right |\ll 1, \ \ k \in \mathbb{Z},
\end{align*}
together with the Hamiltonian equation
\begin{align*}
\dot{\theta} = \frac{\partial H}{\partial p_{\theta}} = \frac{p_{\theta}}{ml^2} ,\ \ \ \ 
\dot{p}_{\theta} = -\frac{\partial H}{\partial \theta} = -mgl\theta .
\end{align*}
Furthermore, since it can be deduced that $dH/dt=\left ( \partial H/\partial \theta \right )\dot{\theta}+ \left ( \partial H/\partial p_{\theta} \right )\dot{p_{\theta}}=0$, when $\theta$ is close enough to $2k\pi,k \in \mathbb{Z}$, the total energy $E$ of such system is conserved. Meanwhile, we can see that $V$ is a quadratic function of $\theta$ under the above circumstance, so that the motion of the pendulum is periodic and the period $T$ can be exactly approximated by $2\pi\sqrt{l/g}$.
To this end, it's certain that the single pendulum behaves as an integrable system for $\left |\theta-2k\pi  \right | \ll 1 ,k \in \mathbb{Z}$, presenting us with an opportunity to shift our focus towards studying the action-angle coordinates of these types of systems. Define the action variable $I$ as a constant related to $E$ (for the sake of convenience, we just assume that $I$ equals to $E$ here) and the angle variable $\eta$ as a representation describing the periodicity. To be specific,
\begin{align*}
I=E,\ \ \ \ \eta=\omega t+\phi,\tag{4.10}
\end{align*}
where the frequency $\omega=\sqrt{g/l}$ is a constant and $\phi$ is the initial phase. 

Similar to the previous model, if we add stochastic perturbations (Brownian motions or $\mathrm{L\acute{e}vy}$ processes) to (4.10) as well, then we can obtain the results analogous to Theorem 4.1 and Theorem 4.2. Since $\omega$ is a constant here, the proofs are considerably simplified and, as such, we will omit them here.

\subsection{High-dimensional Integrable Hamiltonian systems with white noise perturbations}
Now, we will discuss the stability of high-dimensional integrable Hamiltonian systems perturbed by white noise, especially the persistence of their invariant tori.  Analogous to the previous contents, we will discuss two cases of perturbations,
with specific forms including Brownian motions and $\mathrm{L\acute{e}vy}$ processes.

\subsubsection{Gaussian noise perturbations}
Consider a $d$-dimensional integrable stochastic  Hamiltonian system described in action-angle coordinates,
\begin{align*}
I_{s} &=I_{0}+\sigma  B^{1}_{s},\ \ \ \
\theta_{s}= \theta_{0}+\int_{0}^{s}\omega\left ( I_{r} \right )dr+\eta B^{2}_{s},
\end{align*}
where $B^{1}$ and $B^{2}$ are two $d$-dimensional pairwise standard independent Brownian motions defined on a filtered probability space $\left ( \Omega,\mathcal{F},\mathcal{F}_{t},\mathbb{P} \right ) $ satisfying the usual conditions, $B^{1}$ and $B^{2}$ are independent of each other, $\sigma$ and $\eta$ are two constants, $\omega=\left ( \omega_{1},\cdots,\omega_{d} \right )^{\mathrm{T}}$ denotes the frequency of a smooth \color{black} Hamiltonian function $H$ and is bounded and continuous. \color{black} We denote by $\mathbb{E}$ the corresponding expectation of $\mathbb{P}$.
In particular, $\theta_{0}$ is written as $\left ( \theta_{01},\cdots,\theta_{0d} \right )^{\mathrm{T}}$. 
\begin{remark}
In the following contents, we'll always follow the notations in the previous sections. For a given vector $z=\left ( z_{1},\cdots,z_{d} \right )^{\mathrm{T}}$, $z^{2}=zz^{\mathrm{T}}$ is a $d\times d$ matrix whose $i$-th diagonal element is expressed as $z_{i}^{2}$ for $i=1,\cdots,d$.
\end{remark}

\setcounter{theorem}{0} 
\renewcommand{\thetheorem}{4.3}
\begin{theorem}
There exist a neighbourhood $\mathcal{O}$ of $\mathbf{0}\in \mathbb{R}^{d}$, $\delta>0$ and $A_{k\delta},\mu \in \mathbb{R}^{d}$, such that for all $n \in \mathbb{N}$, 
\begin{align*}
	\frac{1}{\sqrt{n\delta}}	\left (\sum_{k=1}^{n}\frac{\theta_{k\delta}}{k\delta}-\sum_{k=1}^{h }m_{k\delta}A_{k\delta}  \right ) 
\end{align*} 
converges to a Gaussian distribution $\mathcal{N}\left (\mathbf{0},\left (\int_{\Omega}\omega\left ( x \right )dx  \right )^{2} \right ) $, where $h=h\left ( n \right ) $ is a positive integer and the sequence $\left \{ m_{k} \right \}_{k=1,\cdots,h} $ satisfies $\sum_{k=1}^{h }m_{k\delta}=n\delta$.
\end{theorem}
\begin{proof}
Following from Theorem 4.1, we try to find the correspondence between our preconditions and Theorem 3.2. Firstly, we denote $T$ as $n\delta$ for a positive constant $\delta$ (for simplicity we suppose $\delta=1$). Next, we define  $Q_{n}=\sum_{k=1}^{n}\theta_{k\delta}/k\delta=\sum_{k=1}^{n}\left ( X_{k}+Y_{k} \right )$,
where
\begin{align*}
	&X_{k}=\frac{1}{k}\left ( \theta_{0}+\int_{0}^{k}\omega \left ( I_{r} \right )dr \right ),\ \ Y_{k}=\frac{\eta B_{k}^{2}}{k},\ \   H_{n}=\sum_{k=1}^{n}X_{k},\ \ V_{n}=\sum_{k=1}^{n}Y_{k}.
\end{align*} 
Using the Birkhoff ergodic theorem again (see (4.6)),  
we obtain the convergence of $X_{k}$. Based on the previous illustrations, the characteristic function of $Q_{n}$ can be expressed as 
\begin{align*}
	\mathbb{E}\left ( \mathrm{exp}\left \{ i\left \langle t,Q_{n} \right \rangle \right \} \right ) &=\mathbb{E}\left ( \mathrm{exp}\left \{ i\left \langle t,\sum_{k=1}^{n}\frac{\theta_{k}}{k} \right \rangle \right \} \right )=\prod_{k=1}^{h}c\left ( t \right ) \lambda_{k}\left (t  \right )^{m_{k}} ,
\end{align*}
where $t\in \mathcal{O}$, $c\left ( t \right ) =1$, $d_{n}\left (  t\right )=0 $, $h=h\left (n  \right ) $ is a positive integer, the sequence $\left \{ m_{k} \right \}_{m=1,\cdots,h} $ satisfies $\sum_{k=1}^{h}m_{k}=n$, $\left \{ \lambda_{k} \right \}_{k=1,\cdots,h}$ and $\left \{ Z_{k} \right \}_{k=1,\cdots,h}$ are denoted as different characteristic functions and random variables of sequence $\left \{ X_{k}+Y_{k} \right \}_{k=1,\cdots,n}$. To be specific, $\lambda_{k}\left ( t \right ) $ is defined as
\begin{align*}
	\lambda_{k}\left ( t \right ) =\mathbb{E}\left ( \mathrm{exp}\left \{ i\left \langle t,Z_{k} \right \rangle  \right \}  \right )
	=\mathrm{exp}\left \{ i\left \langle t,A_{k} \right \rangle-\frac{t^{\mathrm{T}}\sigma_{k}^{2}t}{2} +o\left (t^{\mathrm{T}}t \right )  \right \} ,\ \ t \in \mathcal{O},
\end{align*}
where $A_{k}=\mathbb{E}Z_{k}\in\mathbb{R}^{d}$ and  $\sigma_{k}^{2}=\mathbb{E}Z_{k}^{2}$ is a $d \times d$ matrix. Therefore, we've successfully constructed (i), (ii) and (iii) of Theorem 3.2. Now we turn to prove either (iv) or (v) of Theorem 3.2. Denote $\theta_{k}$ as $\left ( \theta_{k1},\cdots,\theta_{kd} \right )^{\mathrm{T}}$ and every $\theta_{ki}$ as $X_{ki}+Y_{ki}$ for $i=1,\cdots,d$, such that $X_{k}=\left ( X_{k1},\cdots,X_{kd} \right )^{\mathrm{T}}$ and $Y_{k}=\left (Y_{k1},\cdots,Y_{kd}  \right )^{\mathrm{T}} $.
	On the one hand,  given the independence of $B_{k}^{2}$,  it follows that
\begin{align*}
	\mathbb{E} Y_{kp}Y_{kq} =\mathbb{E}Y_{kp}\cdot\mathbb{E}Y_{kq}=0\ \ \mathrm{for}\ \ \ p,q =1,\cdots,d\ \mathrm{and}\ p\ne q,
\end{align*}
which indicates that $\mathbb{E}Y_{k}^{2}$ is a diagonal matrix. According to Theorem 4.1 and (4.7),
\begin{align*}
	\lim_{n \to \infty}\frac{1}{n}\sum_{k=1}^{n}\mathbb{E}Y_{ki}^{2}=0.\tag{4.11}
\end{align*}
On the other hand, 
as for matrix $\mathbb{E}X_{k}^{2}$, we'll study its diagonal and non-diagonal elements separately. Specifically, for the diagonal elements, following from Theorem 4.1, we obtain 
\begin{align*}
	\lim_{k \to \infty}\mathbb{E}X_{ki}^{2}&=\lim_{k \to \infty}\mathbb{E}\left (  \frac{\theta_{0i}}{k}+\frac{1}{k}\int_{0}^{k}  \omega_{i}\left ( I_{r} \right )dr   \right )^{2}=\left (\int_{\Omega} \omega_{i}\left ( x\right )dx   \right )^{2},\ i=1,\cdots,d,
\end{align*}
so that for $i=1,\cdots,d $,
\begin{align*}
	\lim_{n\to\infty}\frac{1}{n}\sum_{k=1}^{n}\mathbb{E}X_{ki}^{2}=\left (\int_{\Omega} \omega_{i}\left ( x\right )dx   \right )^{2}.\tag{4.12}			
\end{align*}
 Now we move on to the non-diagonal elements. Since for different $p,q=1,\cdots,d$, 
\begin{align*}
	\lim_{k\to \infty}\mathbb{E}\left (  \frac{\theta_{0p}}{k} \right ) =\lim_{k \to \infty}\mathbb{E}\left (  \frac{\theta_{0p}}{k} \right ) \cdot \left ( \frac{1}{k}\int_{0}^{k}  \omega_{q}\left ( I_{s} \right )ds   \right )=0,
\end{align*}
then it can be deduced that
\begin{align*}
	\lim_{k \to \infty}\mathbb{E}X_{kp}X_{kq}&=\lim_{k \to \infty}\mathbb{E}\left [   \frac{1}{k}\left (\theta_{0p} +\int_{0}^{k}\omega_{p}\left ( I_{r} \right ) dr \right )+\frac{1}{k}\left (\theta_{0q} +\int_{0}^{k}\omega_{q}\left ( I_{r} \right ) dr \right )\right ]\\
	&=\left ( \int_{\Omega} \omega_{p}\left ( x\right )dx  \right ) \cdot \left ( \int_{\Omega} \omega_{q}\left ( x\right )dx   \right ),
\end{align*}
that is,
\begin{align*}
	\lim_{n \to \infty}\frac{1}{n}\sum_{k=1}^{n}\mathbb{E}X_{kp}X_{kq}=\left ( \int_{\Omega} \omega_{p}\left ( x\right )dx  \right ) \cdot \left (\int_{\Omega} \omega_{q}\left ( x\right )dx   \right ).\tag{4.13}
\end{align*}
As we have already claimed, $Z_{k}$ has the form $\theta_{k}/k=X_{k}+Y_{k}$ and $X_{k}$ is exactly independent of $Y_{k}$, hence,  
we obtain from (4.11), (4.12) and (4.13) that for $p,q=1,\cdots,d$, the matrix $\mu^{2}\triangleq\lim_{n \to \infty}\sum_{k=1}^{h}m_{k}\sigma_{k}^{2}/n$ is well-defined with the elements
\begin{align*}
	\mu_{pq}^{2}&=\left ( \int_{\Omega} \omega_{p}\left ( x\right )dx  \right ) \cdot \left (  \int_{\Omega} \omega_{q}\left ( x\right )dx   \right ),\ \ p,q=1,\cdots,d.\tag{4.14}
\end{align*}
It can be seen that $\mu^{2}$ is precisely equal to $\left (\int_{\Omega}\omega\left ( x \right )dx  \right )^{2}$ which aligns with (iv) of Theorem 3.2. We're now prepared to apply this theorem and reach our result. 
\end{proof}

\begin{remark} The aforementioned theorem demonstrates that, compared to the original systems without perturbations, the action variables $I_{T}$ of the integrable stochastic Hamiltonian systems strictly obey $N\left ( \mu_{1},\sigma_{1}^{2} \right ) $, while the corresponding frequencies obey $N\left ( \mu_{2},\sigma_{2}^{2} \right ) $ asymptotically. Here, $\mu_{i}$ and $\sigma_{i}$, $i=1,2$ depend on the specific systems being studied. We thus have a precise description of the persistence of invariant tori in integrable Hamiltonian systems under white noise perturbations within the framework of the central limit theorems.
\end{remark}

\subsubsection{Non-Gaussian noise perturbations}
Consider a $d$-dimensional integrable stochastic  Hamiltonian system described in action-angle coordinates
\begin{align*}
I_{s}=I_{0}+\sigma  L^{1}_{s}, \ \ \ \
\theta_{s}= \theta_{0}+\int_{0}^{s}\omega\left ( I_{r} \right )dr+\eta L^{2}_{s}, \tag{4.15}
\end{align*}
where 
$L^{1}$ and $L^{2}$ are two independent $\mathrm{L\acute{e}vy}$ noise corresponding to $\left ( \gamma_{1},\xi_{1},N_{1},\nu_{1} \right ) $ and $\left (  \gamma_{2},\xi_{2},N_{2},\nu_{2}\right ) $, $\sigma$ and $\eta$ are two constants, $\omega=\left ( \omega_{1},\cdots,\omega_{d} \right )^{\mathrm{T}}$ denotes the frequency of a smooth \color{black} Hamiltonian function $H$ and is bounded and continuous. Then for $s \ge 0$, (4.15) can be equivalently rewritten in the following form
\begin{align*}
\left\{
\begin{aligned}
	I_{s} &=I_{0}+\sigma \gamma_{1} s+\sigma \xi_{1}\color{black}B^{1}_{s}+\sigma\int_{\left |z  \right |<1 } z\tilde{N}_{1}\left ( s,dz \right ) ,\\
	\theta_{s}&= \theta_{0}+\int_{0}^{s}\omega\left ( I_{r}  \right )dr+\eta\gamma_{2} s+\eta\xi_{2} B^{2}_{s}+\eta\int_{\left |z  \right |<1 }z\tilde{N}_{2}\left ( s,dz \right ),
\end{aligned}
\right.
\end{align*}
where $\int\min\left \{\left |z  \right |^{2} ,1  \right \} \nu_{i}\left ( dz \right )  <\infty$ for $i=1,2$ and $z=\left ( z_{1},\cdots,z_{d} \right )^{\mathrm{T}} $. Moreover, the assumptions for Brownian motions and their probability spaces remain identical to those outlined in Theorem 4.3.

\setcounter{theorem}{0} 
\renewcommand{\thetheorem}{4.4}
\begin{theorem}
There exist a neighbourhood $\mathcal{O}$ of $\mathbf{0}\in \mathbb{R}^{d}$, constant $\delta>0$, $W_{k},\mu\in\mathbb{R}^{d}$, such that $\left (\sum_{k=1}^{n}\theta_{k\delta}/k\delta-\sum_{k=1}^{h }m_{k\delta}A_{k\delta}  \right ) /\sqrt{n\delta}$ converges to a Gaussian distribution $\mathcal{N}\left (\mathbf{0},\zeta^{2} \right ) $ for all $n \in \mathbb{N}$, where $h=h\left ( n \right ) $ is a positive integer,  $\sum_{k=1}^{h }m_{k\delta}=n\delta$ and
the $\left ( p,q \right ) $-th element of matrix $\zeta^{2}$ has a lower bound (4.14) and an upper bound defined as the sum of (4.14) and (4.17) for $p,q=1,\cdots,d$.
\end{theorem}
\begin{proof}
Similar to the proof of Theorem 4.3, for some constant $\delta>0$, which can be normalized as $1$ , we denote $T$ as $n\delta$ and divide $\left [ 0,T \right ] $ into $n$ equidistance parts.
By defining $Q_{n}=\sum_{k=1}^{n}\theta_{k}/k$, we pick sequences $H_{n}$, $V_{n}$, such that $Q_{n}=H_{n}+V_{n}$ and
\begin{align*}
	H_{n}&=\sum_{k=1}^{n}\frac{1}{k}\left ( \theta_{0}+\int_{0}^{k}\omega\left ( I_{s} \right ) ds+\eta \xi_{2}B^{2}_{k} \right )=\sum_{k=1}^{n}X_{k} ,\\
	V_{n}&=\sum_{k=1}^{n}\frac{1}{k}\left ( \eta\gamma_{2}k+\eta\int_{\left |z  \right |<1 }z\tilde{N}_{2}\left ( k,dz \right )  \right )=\sum_{k=1}^{n}Y_{k}.
\end{align*}
Here we denote $B_{k}^{2}$ as $\left ( B_{k1}^{2},\cdots,B^{2}_{kd} \right )^{\mathrm{T}} $. Therefore, given the independence of $\theta_{p}$ and $\theta_{q}$ for different $p,q=1,\cdots,d$, it follows that
\begin{align*}
	\mathbb{E}\left ( \mathrm{exp}\left \{i\left \langle t,Q_{n} \right \rangle  \right \} \right ) =\mathbb{E}\left ( \mathrm{exp}\left \{ i\left \langle t,\sum_{k=1}^{n}\frac{\theta_{k}}{k} \right \rangle \right \}  \right ) =\prod_{k=1}^{n}\mathbb{E}\left ( \mathrm{exp}\left \{ i\left \langle t,\frac{\theta_{k}}{k} \right \rangle \right \} \right ) .\tag{4.16}
\end{align*}
Applying the characteristic function grouping idea mentioned before, we choose an integer $h=h\left ( n \right )>0$ and a sequence $\left \{m_{k}  \right \}_{k=1,\cdots,h} $ satisfying $\sum_{k=1}^{h}m_{k}=n$, then we can divide $\left \{ \theta_{j}/j \right \}_{j=1,\cdots,n}$ into $h$ groups with different characteristic functions and represent the $k$-th "group" with characteristic function $\lambda_{k}\left ( t \right )$ and $m_{k}$ numbers of variables from the sequence $\left \{ \theta_{j}/j \right \}_{j=1,\cdots,n}$, which are denoted as $Z_{k},k=1,\cdots,h$. To this end, (4.16) turns to
\begin{align*}
	\mathbb{E}\left ( \mathrm{exp}\left \{ i\left \langle t,Q_{n} \right \rangle \right \} \right ) =\prod_{k=1}^{h}\lambda_{k}\left (t  \right )^{m_{k}},\ \ t \in \mathcal{O},
\end{align*}
where $\lambda_{k}\left ( t \right )$ can be expressed as
\begin{align*}
	\lambda_{k}\left ( t \right )=\mathbb{E}\left ( \mathrm{exp}\left \{ i\left \langle t,Z_{k} \right \rangle  \right \}  \right )
	=\mathrm{exp}\left \{ i\left \langle t,\mathbb{E}Z_{k} \right \rangle-\frac{t^{\mathrm{T}}\mathbb{E}Z_{k}^{2}t}{2} +o\left ( t^{\mathrm{T}}t \right )  \right \},\ \ t \in \mathcal{O} .
\end{align*}
If $A_{k},\sigma_{k}\in \mathbb{R}^{d}$ and functions $c\left ( t \right )$, $d_{n}\left ( t \right ) $ are defined as $A_{k}=\mathbb{E}Z_{k} $, $ \sigma_{k}^{2}=\mathbb{E}Z_{k}^{2}$ and $c\left ( t \right )=1 $, $d_{n}\left ( t \right )=0 $, then 
\begin{align*}
	\mathbb{E}\left ( \mathrm{exp}\left \{ i\left \langle t,Q_{n} \right \rangle \right \} \right )
	&=\prod^{h}_{k=1}c\left ( t \right ) \lambda_{k}\left ( t \right )^{m_{k}}+d_{n}\left (t  \right )  ,
\end{align*}
which corresponds to (i) of Theorem 3.2. Moreover, by examining the continuity of $c=1$ at $t=\mathbf{0}$ and the convergence of $\left \| d_{n}\left ( t \right ) \right \|_{L^{\infty}\left ( \mathcal{O} \right ) }$ towards 0 as $n$ goes to infinity, it remains  to verify (iv) or (v) of Theorem 3.2.
According to Theorem 4.3, the matrix $\lim_{n \to \infty}\sum_{k=1}^{n}\mathbb{E}X_{k}^{2}/n$ is well defined as $\left ( \int_{\Omega} \omega\left ( x \right ) dx \right )^{2} $. Now we pay our attention to the additional part.  If we denote $Y_{k}=\left ( Y_{k1},\cdots,Y_{kd} \right )^{\mathrm{T}}$ and $\gamma_{2}=\left ( \gamma_{21},\cdots,\gamma_{2d} \right )^{\mathrm{T}}$, then according to Theorem 4.2, for $i=1,\cdots,d$, the diagonal part of $\lim_{n\to\infty}\sum_{k=1}^{n}\mathbb{E}Y_{k}^{2}/n$ can be estimated as
\begin{align*}
	\lim_{n\to\infty}\frac{1}{n}\sum_{k=1}^{n}\mathbb{E}Y_{ki}^{2}\color{black}&
	\le\lim_{n\to\infty}\frac{1}{n}\sum_{k=1}^{n} \left [ 2\eta^{2}\gamma_{2i}^{2}+ \frac{2\eta^{2}}{k}\int_{\left |z_{i}  \right |<1 }\left |z_{i}  \right |^{2}\nu\left ( dz_{i} \right )\right ]  =2\eta^{2}\gamma_{2i}^{2}  .
\end{align*}
Meanwhile, for different $p,q=1,\cdots,d$, 
\begin{align*}
	\mathbb{E}Y_{kp}Y_{kq}&= \mathbb{E}\left ( \eta\gamma_{2p}+\frac{\eta}{k}\int_{\left |z_{p}  \right |<1 }z_{p}\tilde{N}_{2}\left ( k,dz_{p} \right )  \right )\left ( \eta\gamma_{2q}+\frac{\eta}{k}\int_{\left |z_{q}  \right |<1 }z_{q}\tilde{N}_{2}\left ( k,dz_{q} \right )  \right ) \\
	&\le \frac{\eta^{2}}{2}\left ( \gamma_{2p}+\gamma_{2q}\right )^{2}+\frac{\eta^{2}}{k} \int_{\left |z_{p}  \right |<1 }\ \left | z_{p} \right |^{2}  \nu\left ( dz_{p} \right )+\frac{\eta^{2}}{k} \int_{\left |z_{q}  \right |<1 }\ \left | z_{q} \right |^{2}  \nu\left ( dz_{q} \right ),
\end{align*}
so that
\begin{align*}
	\lim_{n \to \infty}\frac{1}{n}\sum_{k=1}^{n}\mathbb{E}Y_{kp}Y_{kq}\le \frac{\eta^{2}}{2}\left ( \gamma_{2p}+\gamma_{2q}\right )^{2}.
\end{align*}
Therefore, for every $p,q=1,\cdots,d$, we conclude that
\begin{align*}
\lim_{n \to \infty}\frac{1}{n}\sum_{k=1}^{n}\mathbb{E}Y_{kp}Y_{kq}&\le\frac{\eta^{2}}{2}\left ( \gamma_{2p}+\gamma_{2q}\right )^{2} .
\end{align*}
To this end, utilizing the independence between $X_{k}$ and $Y_{k}$, we construct a matrix $\eta^{2}=\left ( \eta_{pq}^{2} \right )_{d\times d} $ for $\lim_{n \to \infty}\sum_{k=1}^{h}m_{k}\sigma_{k}^{2}/n$ such that 
\begin{align*}
	\eta_{pq}^{2}=\eta^{2}\left ( \gamma_{2p}+\gamma_{2q}\right )^{2} +2\left ( \int_{\Omega}\omega_{p}\left ( x \right )dx  \right )\left ( \int_{\Omega}\omega_{q}\left ( x \right )dx  \right ) ,\ \ p,q=1,\cdots,d,\tag{4.17}
\end{align*}
while the $\left ( p,q \right ) $-th element of matrix $\lim_{n \to \infty}\sum_{k=1}^{h}m_{k}\sigma_{k}^{2}/n$ has a lower positive bound (4.14), that is, $\mu_{pq}^{2}$, and an upper bound (4.14)+(4.17), that is, $\eta^{2}_{pq}+\mu_{pq}^{2}$. We're now in a position to apply Theorem 3.2 and reach our final result.
\end{proof}

\section*{Acknowledgements}
\begin{sloppypar}
The second author (Li Yong) is supported by National Basic Research Program of China (Grant number
[2013CB8-34100]), National Natural Science Foundation of China (Grant numbers [11571065], [11171132], and
[12071175]),
and Natural Science Foundation of Jilin Province (Grant number  [20200201253JC]).
\end{sloppypar}

\end{document}